\documentclass{article}
\author{
Mikhail Ganzhinov\\
Department of Communications and Networking\\
Aalto University School of Electrical Engineering\\
P.O.\ Box 15400, 00076 Aalto, Finland\\
}
\date{}
\usepackage{graphicx}
\usepackage{float}
\usepackage[margin=1.0in]{geometry}
 \usepackage{amsfonts}
\usepackage{amsmath}
\usepackage{amsthm}
\usepackage{xcolor}
\usepackage{hyperref}
\newtheorem{theorem}{Theorem}
\newtheorem{proposition}{Proposition}
\newtheorem{definition}{Definition}
\newtheorem{lemma}{Lemma}
\newtheorem{corollary}{Corollary}
\newtheorem{question}{Question}

\theoremstyle{remark}
\newtheorem{remark}{Remark}
\newcommand{\SL}{\textrm{SL}}
\newcommand{\GL}{\textrm{GL}}
\newcommand{\Det}{\textrm{det}}
\newcommand{\Tr}{\textrm{Tr}}
\title{Infinite families of optimal systems of biangular lines related to representations of $\SL(2,\mathbb{F}_q)$.}
\begin{document}
\maketitle
\begin{abstract}
A line packing is optimal if its coherence is as small as possible. Most interesting examples of optimal line packings are achieving equality in some of the known lower bounds for coherence. In this paper two infinite families of real and complex biangular line packings are presented. New packings achieve equality in the real or complex second Levenshtein bound respectively. Both infinite families are constructed by analyzing well known representations of the finite groups $\SL(2,\mathbb{F}_q)$. Until now the only known infinite familes meeting the second Levenshtein bounds were related to the maximal sets of mutually unbiased bases (MUB). Similarly to the line packings related to the maximal sets of MUBs, the line packings presented here are related to the maximal sets of mutually unbiased weighing matrices. Another similarity is that the new packings are projective 2-designs. The latter property together with sufficiently large cardinalities of the new packings implies some improvement on largest known cardinalities of real and complex biangular tight frames.
\end{abstract}
\bigskip
\section{Introduction}
The Grassmannian line packing problem asks for line packings with minimal coherence. Packings with such properties are useful in signal processing applications \cite{St} as they lead to improved error correction. The exact definition of coherence may vary in different applications. 

Let $\mathcal{L}$ be a set of lines represented by the set of unit vectors $\Phi_{\mathcal{L}}\subset\mathbb{F}^n$ where the field $\mathbb{F}$ is $\mathbb{R}$ or $\mathbb{C}$.  The angle set of the line packing $\mathcal{L}$ is defined as $\mathcal{A}_\mathcal{L}=\{|\langle x,y\rangle|: x,y\in \Phi_\mathcal{L},x\neq y\}$. One natural way to define coherence $\mu_{\mathcal{L}}$ of the set $\mathcal{L}$ is to put
\begin{displaymath}
\mu_{\mathcal{L}}=\underset{\alpha\in\mathcal{A}_{\mathcal{L}}}{\textrm{max}}\alpha. 
\end{displaymath}
Line packings achieving minimal coherence among all packings of the same size in $\mathbb{F}^n$ are called \emph{optimal} or \emph{Grassmannian line packings}. For $|\mathcal{A}_\mathcal{L}|=1,2\textrm{ and }k$ corresponding line packings are called \emph{equiangular}, \emph{biangular} and \emph{k-angular}, respectively.

Other properties of line packings may also be of great interest. A finite set of unit norm vectors $\Phi\subset\mathbb{F}^n$ is a \emph{frame} if $\Phi$ spans $\mathbb{F}^n$. A frame $\Phi$ is \emph{tight} if there exists a constant $a>0$ such that $av=\sum_{\phi\in\Phi}\langle v,\phi\rangle\phi$ for any $v\in\mathbb{F}^n$. In other words, tight frames generalize some aspects of orthogonal bases. If $\Phi_\mathcal{L}$ is a tight frame, then the packing $\mathcal{L}$ is also called a tight frame, in which case $a=|\mathcal{L}|/n$. The tightness of a frame is usually desired along other properties such as optimality and a small angle set.

Several lower bounds for minimal coherence are known. Packings with equality in those bounds are guaranteed to be optimal. One of these bounds is the Welch bound \cite{Welch} stating that
\begin{displaymath}
\mu_\mathcal{L}\geq\sqrt{\frac{|\mathcal{L}|-n}{n(|\mathcal{L}|-1)}}
\end{displaymath}
for any $n$-dimensional set of lines $\mathcal{L}$. The Welch bound coincides with the first Levenshtein bound for line packings \cite{L}. Line packings with equality in this bound are exactly equiangular tight frames (ETF). These are probably the most studied line systems \cite{FF,FFF,GroupPackings,Su}. The maximal cardinality of real or complex $n$-dimensional equiangular line packings cannot exceed $n+1\choose2$ and $n^2$ lines, respectively \cite{D}. Zauner's conjecture \cite{Z} predicts existence of complex ETFs of maximal cardinality $n^2$ in every dimension. According to \cite{Conj,Conj1} existence of these packings seem to be related to abelian extentions of real quadratic fields. 

Packings in $n$-dimensional space with cardinality larger than $n+1\choose 2$ in the real case and $n^2$ in the complex case cannot be equiangular and cannot achieve equality in the Welch bound. To prove optimality of line packings with cardinality in the range ${n+1\choose 2}<|\mathcal{L}|\leq{n+2\choose 3}$ in the real case and $n^2<|\mathcal{L}|\leq n{n+1\choose 2}$ in the complex case, it is sufficient to achieve equality in the second Levenshtein bound \cite{L}. If an $n$-dimensional line system $\mathcal{L}$ is real, then the Levenshteins second bound gives that 
\begin{displaymath}%\label{LSB1}
\mu_\mathcal{L}\geq\sqrt{\frac{3|\mathcal{L}|/2-n(n/2+1)}{(|\mathcal{L}|-n)(n/2+1)}},
\end{displaymath}
while in the complex case we have
\begin{displaymath}%\label{LSB2}
\mu_\mathcal{L}\geq\sqrt{\frac{2|\mathcal{L}|-n(n+1)}{(|\mathcal{L}|-n)(n+1)}}.
\end{displaymath}
Line systems achieving equality in the Levenshtein's second bound are necessarily $\{0,\mu_\mathcal{L}\}$-angular projective 2-designs. Real line systems have the structure of $Q$-polynomial three class association schemes \cite{B}. Such line systems are rare. In fact, we know only finitely many constructions achieving equality in the second Levenshtein's bound that aren't either maximal sets of equiangular lines or maximal sets of mutually unbiased bases, see \cite{Haa} for a complete list. The list of known biangular projective 2-designs (i.e., line systems achieving equality in biangular relative bounds \cite{D,Hog}) is slightly longer, but remains finite. Biangular tight frames (biangular projective 1-designs) are investigated in \cite{Casazza1,Casazza} and are more numerous.  Despite that, constructing biangular tight frames with cardinalities exceeding  the upper bound for equiangular lines proves to be challenging. To underline this fact, the following question is posed in \cite{Casazza1}. 
\begin{question}
\label{Q1}
Are  there other ``infinite families" of complex biangular tight frames in addition to maximal sets of mutually unbiased bases \cite{Woo} with cardinalities exeeding square of dimension of the space they reside?
\end{question}
In this paper we construct real $\{0,\sqrt{2/q}\}$-angular line systems of cardinality $q^2-1$ in $\mathbb{R}^{q-1}$ when $q$ is an odd power of $2$. We also construct complex $\{0,\sqrt{3/q}\}$-angular line systems of cardinality $(q^2-1)/2$ in $\mathbb{C}^{(q-1)/2}$ when $q$ is a power of $3$. Both constructions produce infinite families of systems of biangular lines which achieve equality in the Levenshtein's second bound. These are the first infinite families unrelated to mutually unbiased bases. The cardinalities of the line systems in both families are large enough to give positive answer to Question \ref{Q1}. Interestingly, the new line systems are related to mutually unbiased weighing bases \cite{Be}.

The paper is organized as follows. In Section 2 we describe the mathematical machinery and the notations that will be used through the rest of the paper. In Section 3 we construct well-known $(q-1)$-dimensional representations of $\SL(2,\mathbb{F}_q)$. In Section 4 we  construct real infinite family of optimal line systems related to certain $2^{2k+1}-1$-dimensional representations of $\SL(2,\mathbb{F}_{2^{2k+1}})$. Finally, in Section 5 we  construct complex infinite family of optimal line systems related to $(3^k-1)/2$-dimensional representations of $\SL(2,\mathbb{F}_{3^k})$.
\bigskip
\section{Preliminaries}

Let $\mathbb{F}_q$ be a finite field of order $q=p^n$ where $p$ is a prime and $n\in\mathbb{N}$. Let $\mathbb{L}/\mathbb{F}_q$ be a quadratic field extension of $\mathbb{F}_q$. If $q$ is odd, $\mathbb{L}$ can be defined as $\mathbb{F}_q(\rho)$ where $\rho\notin\mathbb{F}_q$ is a root of some quadratic polynomial of the form  $x^2-s$, $s\in\mathbb{F}_q$, that is irreducible over $\mathbb{F}_q$. If $q$ is even,  instead of $x^2-s$, an irreducible  polynomial  (over $\mathbb{F}_q$) of the form $x^2+x+s$ can be used. 

Every element $z$ of the field  $\mathbb{L}$ can be expressed in the form $z=x+y\rho$ and thus the additive group of $\mathbb{L}$ is isomorphic to the vector space $\mathbb{F}_q^2$. Let $\mathbb{F}_q^*$ and $\mathbb{L}^{*}$ be the multiplicative groups of the fields $\mathbb{F}_q$ and $\mathbb{L}$, respectively; these groups are cyclic. If $q$ is odd, $\mathbb{F}_q^*$ can be decomposed into two sets of size $(q-1)/2$, $QR_q$ and $NQR_q$, containing the squares (quadratic residues) and the nonsquares (quadratic nonresidues), respectively.

The field $\mathbb{L}$ has a unique nontrivial automorphism $\sigma$ fixing $\mathbb{F}_q$. If $q$ is odd, then $\sigma(x+y\rho)=x-y\rho$ while if $q$ is even then $\sigma(x+y\rho)=x+y+y\rho$ for all $x,y\in\mathbb{F}_q$. In both cases
\begin{displaymath}
\sigma(z)=z^q\textrm{ for all }z\in\mathbb{L}.
\end{displaymath}
The functions $N:\mathbb{L}\rightarrow\mathbb{F}_q$ and $S:\mathbb{L}\rightarrow\mathbb{F}_q$  defined by 
\begin{displaymath}
N(z)=z\sigma(z)=z^{q+1}\textrm{ and }S(z)=z+\sigma(z)=z+z^q
\end{displaymath}
for all $z\in\mathbb{L}$  are \emph{relative} \emph{norm} and \emph{relative trace}, respectively. A set
\begin{displaymath}
C_d=\{t\in\mathbb{L}\mid N(t)=d\}\textrm{ where }d\in\mathbb{F}_q^*
\end{displaymath}
is called a \emph{circle} of the field $\mathbb{L}$. If $d=1$, the notation $C_1=C$ will be used. Each circle consists of $q+1$ elements, and $C$ is a cyclic subgroup of $\mathbb{L}^*$. The \emph{absolute trace}  of the field $\mathbb{F}_q$ is the function $\Tr:\mathbb{F}_q\rightarrow\mathbb{F}_p$ defined by
\begin{displaymath}
\Tr(x)=x+x^p+\ldots+x^{p^{n-1}} \textrm{ for all } x\in\mathbb{F}_q.
\end{displaymath}
It is well known that while the norm is a multiplicative function, the trace is additive.

Let $A$ be an abelian group. A function $\rho:A\rightarrow\mathbb{C}\setminus\{0\}$ is called a \emph{character} of the group $A$ if it is a homomorphism. The set $\hat{A}$ formed by all characters of $A$ is a group under pointwise multiplication. The group $\hat{A}$ is always isomorphic to $A$. The character $\rho_0$ is said to be trivial if $\rho_0(a)=1$ for all $a\in A$. The order of a character $\rho$ is defined to be the smallest integer $k$ such that $\rho^k=\rho_0$.
Additive characters of the field $\mathbb{F}_q$ are homomorphisms from the additive group of the field $\mathbb{F}_q$ to $\mathbb{C}\setminus\{0\}$. The nontrivial character $\chi$ of $\mathbb{F}_q$ can be defined by following equation
\begin{displaymath}
\chi(x)=e^{2\pi i\Tr(x)/p} \textrm{ for all }x\in\mathbb{F}_q.
\end{displaymath}
Similarly, multiplicative characters of a group $A\leq\mathbb{F}_q^*$ are homomorphisms from $A$ to $\mathbb{C}\setminus\{0\}$.

A delta function $\delta_c:\mathbb{F}\rightarrow\mathbb{R}$, where $\mathbb{F}$ is either a field or a multiplicative group of a field (depending on the context) and $c\in\mathbb{F}$ is defined as
%$\chi(x)=(-1)^{Tr(x)}$\textcolor{green}{
\begin{displaymath}
\delta_c(x)=\begin{cases}
1,&\text{if } x=c\\
0,&\text{otherwise}
\end{cases}
\end{displaymath}
for all $x\in\mathbb{F}$. If $c=0$, we simply write $\delta$ instead of $\delta_0$.

The special linear group $G=\SL(2,\mathbb{F}_q)$ is a set of matrices
$\{
\arraycolsep=1pt\def\arraystretch{1}g=\left[
\scriptsize
\begin{array}{cc}
a&b\\
c&d
\end{array}
\right]
\mid a,b,c,d\in\mathbb{F}_q, \textrm{ det}(g)=1\}$ with the group operation being matrix multiplication. The group formed by invertible upper triangular matrices $\arraycolsep=1pt\def\arraystretch{1}g=\left[
\scriptsize
\begin{array}{cc}
a&b\\
0&d
\end{array}
\right]$ where $a,d\in\mathbb{F}_q^*$ and $b\in\mathbb{F}_q$  will be denoted by $B$.

Finally, the Hilbert space of real or complex (depending on the context) functions over the set $S$  is  denoted by $H(S)$ and the associated norm of a vector $v\in H(S)$  is  denoted by $||v||_{H(S)}$.

\bigskip
\section{Some representations of finite groups $\SL(2,\mathbb{F}_q)$}

In representation theory, elements of groups are represented as linear transformations of vector spaces. A representation of a group $G$ is a pair $(V,\phi)$ where $V$ is a vector space and $\phi:G\rightarrow\GL(V)$ is a group homomorphism. In this paper we will be interested in $(q-1)$-dimensional representations of $\SL(2,\mathbb{F}_q)$ and related line systems defined by the row and column vectors of matrices $\phi(g)$ where $g\in\SL(2,\mathbb{F}_q)$. Throughout this section we will use $G$ to denote specifically $\SL(2,\mathbb{F}_q)$.

 All irreducible representations of the groups $\SL(2,\mathbb{F}_q)$ with $q$ odd are covered (for example) in \cite{repr}. Throughout this section, the paper \cite{repr} will be used as the main source of information regarding representations of $\SL(2,\mathbb{F}_q)$. Ultimately, some of these representations will produce optimal line packings.

 As a starting point we introduce certain reducible representations of the groups $\SL(2,\mathbb{F}_q)$ which later will be decomposed into irreducible ones. Following \cite{repr} we are going to define matrices $T_{g}:H(\mathbb{L})\rightarrow H(\mathbb{L})$ for every $g\in\SL(2,\mathbb{F}_q)$ entrywise by
\begin{equation}\label{eq:1}
T_g(u,v)=\begin{cases}
-\frac{1}{q}\chi(\frac{aN(u)+dN(v)-S(u\sigma(v))}{c}),&\text{if } g\notin B\\
\chi(abN(u))\delta(v-ua),&\text{otherwise,}
\end{cases}
\end{equation}
where  $g=\left[
\scriptsize
\begin{array}{cc}
a&b\\
c&d
\end{array}
\right]$, $a,b,c,d\in\mathbb{F}_q$ and $u,v\in\mathbb{L}$ (in \cite{repr} the notation $K^{(2)}(g|u,v)$ is used for the same matrices). We shall now show that the matrices $T_g$ form a group.
\begin{proposition}
The matrices $T_g$ are unitary and form a group $T$. The group $T$ is isomorphic to $G$ with the group isomorphism $\phi:G\rightarrow T$ defined by $\phi(g)=T_g$ for all $g\in G$. 
\end{proposition}
\begin{proof}
For the case where $q$ is odd, see \cite[pp. 70-72]{repr}. If $q$ is even, then the matrices $T_g$ are real. The group property is equivalent to the equality 
\begin{equation}
\label{groupProp}
\langle T_{g_1}(u,:),T_{g_2}(:,v)^{T}\rangle=T_{g_1g_2}(u,v) \textrm{ for all } g_1,g_2\in G \textrm{ and } u,v\in\mathbb{L}. 
\end{equation}
When verifying the group property several cases arise. While we will skip some technical details, we will give the general ideas how to prove each of the arizing cases. 

If at least one of the matrices $g_i$ is an element of the group $B$, the verification of the group property is simple since one of the matrices will be monomial, i.e, in each row and column there is exactly one non-zero entry.

If $g_1g_2=1$, then verification of the group property will only involve evaluation of the linear character sums. Verification of orthogonality of each matrix $T_g$ also involves only evaluation of the linear character sums.

The most difficult case is when $g_1,g_2\notin B$. Now the problem can be reduced to computation of quadratic character sums of the form
\begin{displaymath}
I(v)=\sum_{u\in\mathbb{L}}\chi(N(u)+S(v\sigma(u))).
\end{displaymath}
If $u=x+y\rho$ and $v=z+w\rho$, then the sum $I(v)$ can be rewritten as
\begin{equation}\label{eq:sum}
I(v)=\sum_{x,y\in\mathbb{F}_q}\chi(x^2+xy+sy^2+xw+yz).
\end{equation}
To evaluate $I(0)$ observe that when $q$ is even the function $f:\mathbb{F}_q\rightarrow\mathbb{F}_q$, $f(x)=x^2$ for all $x\in\mathbb{F}_q$ is a bijection and thus $\sum_{x\in\mathbb{F}_q}\chi(x^2)=\sum_{x\in\mathbb{F}_q}\chi(x)=0$. Each circle $C_d$ contains $q+1$ points one of which is on $\mathbb{F}_q$ namely $\sqrt{d}$. Thus, the sum $I(0)$ can be rewritten as
\begin{displaymath}
I(0)=\sum_{u\in\mathbb{L}}\chi(N(u))=\chi(0)+(q+1)\sum_{u\in\mathbb{F}_q^*}\chi(N(u))=(q+1)\sum_{u\in\mathbb{F}_q}\chi(u^2)-q=-q.
\end{displaymath}
All other sums $I(v)$ are related to $I(0)$. By performing a change of variables $u\rightarrow u+v$ in \eqref{eq:sum} we obtain
\begin{displaymath}
I(v)=\sum_{x,y\in\mathbb{F}_q}\chi(x^2+xy+sy^2+z^2+zw+sw^2)=\sum_{u\in\mathbb{L}}\chi(N(v)+N(u))=-\chi(N(v))q.
\end{displaymath}
Now, verification of \eqref{groupProp} will involve either calculation of sums $I(v)$, calculation of linear character sums or the evaluation of row vectors of the matrix $T_{g_1}$ at a single element.
\end{proof}

 A group isomorphism $\phi$ defines a group homomorphism $\phi_1:G\rightarrow\GL(H(\mathbb{L}))$. Thus a pair $(H(\mathbb{L}),\phi_1)$ is a $q^2$-dimensional representation of $G$. The next goal is to decompose a given $q^2$-dimensional representation of $G$ into smaller invariant subspaces. For $t\in C$ define an operator $R_t$ by $R_tf(v)=f(tv)$  for all  $f\in H(\mathbb{L})$ and $v\in\mathbb{L}$. We have
\begin{displaymath}
T_gR_t f(v)=T_g f(tv)=-\sum_{v\in\mathbb{L}}\chi(\frac{aN(u)+dN(v)-S(u\sigma(v))}{c})f(tv)/q.
\end{displaymath}
Now a change of variables $v\to\sigma(t)v$ gives
\begin{displaymath}
T_gR_t f(u)=-\sum_{u\in\mathbb{L}}\chi(\frac{aN(tv)+dN(\sigma(t)u)-S(v\sigma(\sigma(t)u))}{c})f(u)/q=R_tT_g f(u)
\end{displaymath}
for any $f\in H(\mathbb{L})$. Thus, the operators $R_t$ commute with the matrices $T_g$. Let $\pi$ be a character of the group $C$. Let $H(\mathbb{L})_\pi$ be a subspace of $H(\mathbb{L})$ consisting of elements $f$ that satisfy $f(tu)=\pi(t)f(u) \textrm{ for all } t\in C$  and $u\in\mathbb{L}$. For all $f\in H(\mathbb{L})_\pi$
\begin{displaymath}
R_t T_g f=\pi(t)T_g f\in H(\mathbb{L})_\pi
\end{displaymath}
which makes $H(\mathbb{L})_\pi$ an invariant subspace of the representation $\phi_1$.

 If $\pi_0$ is the trivial (constant) character of $C$, then the subspace $H(\mathbb{L})_{\pi_0}$ is $q$-dimensional and consists of functions that are constant on circles $C_d$ and $\{0\}$. If the character $\pi$ is not trivial, the corresponding space $H(\mathbb{L})_{\pi}$ is $(q-1)$-dimensional since necessarily $f(0)=0$ for all  $f\in H(\mathbb{L})_\pi$. Since subspaces $H(\mathbb{L})_\pi$ are clearly orthogonal to each other, the  whole $q^2$-dimensional space $H(\mathbb{L})$ can be decomposed into a direct sum of $q+1$ subspaces $H(\mathbb{L})_\pi$, indeed, $q^2=q+q(q-1)$.

 Let $\theta=\{ x_d \mid x_d\in C_d \textrm{ for every } d\in\mathbb{F}_q^*\}$ be a system of representatives of circles $C_d$. For each nontrivial character $\pi$, an orthonormal basis $B_{\pi,\theta}=\{e_d\}$ of the subspace $H(\mathbb{L})_\pi$ is obtained by choosing
\begin{displaymath}
e_d(x)=\begin{cases}
\frac{1}{\sqrt{q+1}}\pi(x/x_d),&\text{if } x\in C_d\\
0,&\text{otherwise.}
\end{cases}
\end{displaymath}
Let $F_{\pi,\theta}:H(\mathbb{L})_\pi\rightarrow H(\mathbb{F}_q^*)$ be the unitary transformation defined by $F_{\pi,\theta}(e_d)=\delta_d$ for all $e_d\in B_{\pi,\theta}$. Representations associated with subspaces $H(\mathbb{L})_\pi$ could be constructed explicitly as the pairs $(H(\mathbb{F}_q^*),\phi_\pi)$ where $\phi_\pi:G\rightarrow\GL(H(\mathbb{F}_q^*))$ is a group homomorphism. An easy way to accomplish this is to simply put $\phi_\pi(g)=T_{g,\pi,\theta}$ where $T_{g,\pi,\theta}=F_{\pi,\theta}T_{g}F_{\pi,\theta}^{-1}$. Let $T_{\pi,\theta}$ be the set of all matrices $T_{g,\pi,\theta}$ with fixed $\pi$ and $\theta$. Explicitly, after performing matrix multiplications we get
\begin{equation}\label{Erepr}
T_{g,\pi,\theta}(u,v)=\begin{cases}
\frac{-1}{q}\chi(\frac{au+dv}{c})\sum_{t\in C}\chi(\frac{-S(tx_v\sigma(x_u))}{c})\pi(t),&\text{if } g\notin B\\
\pi(ax_u/x_v)\chi(abv)\delta(v-N(a)u),&\text{otherwise}
\end{cases}
\end{equation}
for all $v,u\in\mathbb{F}_q^*$.
\begin{definition}
Character sums of the form
\begin{displaymath}
J_\pi(z):=\frac{1}{q}\sum_{t\in C}\chi(-S(tz))\pi(t) \textrm{ where } z\in\mathbb{L}
\end{displaymath}
are called Bessel functions of the field $\mathbb{L}$.
\end{definition}
By using Bessel functions in the case $g\notin B$, the matrices $T_{g,\pi,\theta}$ can be expressed as
\begin{equation}\label{Brepr}
T_{g,\pi,\theta}(u,v)=-\chi(\frac{au+dv}{c})J_\pi(x_v\sigma(x_u)/c) \textrm{ for all } u,v\in\mathbb{F}_q^*.
\end{equation}

At this point it is good to summarize what has been done so far. We began by introducing $q^2$-dimensional representations $(H(\mathbb{L}),\phi_1)$ where matrices corresponding to elements $g\in\SL(2,\mathbb{F}_q)$ were defined by \eqref{eq:1}. Later, representations $(H(\mathbb{L}),\phi_1)$ were decomposed into invariant subspaces $H(\mathbb{L})_\pi$ where $\pi$ is a character of $C$. The representation corresponding to the subspace $H(\mathbb{L})_{\pi_0}$ where $\pi_0$ is trivial is $q$-dimensional and will not be needed in following sections. All other $(q-1)$-dimensional invariant subspaces $H(\mathbb{L})_\pi$ where $\pi$ is nontrivial were used to construct representations $(H(\mathbb{F}_q^*),\phi_\pi)$. The homomorphism $\phi_\pi$ is given explicitly by \eqref{Erepr} and \eqref{Brepr} through Bessel functions of the field $\mathbb{L}$. Representation $(H(\mathbb{F}_q^*),\phi_\pi)$ depends on the choice of representatives $\theta$. We conclude this section by remarks which clarify how exactly the choice of $\theta$ affects the corresponding representations.
\begin{remark}\label{INV}
It is clear that the particular form of the representations $(H(\mathbb{F}_q^*),\phi_\pi)$ depends on the particular choice of representatives $\theta$. Let $\theta_i=\{x_{d,i}\mid x_{d,i}\in C_d \textrm{ for every } d\in\mathbb{F}_q^*\}$, $i=1,2$ be two different systems of representatives. From the definition of Bessel functions it is easy to see that 
\begin{equation}\label{eq:Bessel}
J_\pi(tz)=\overline{\pi}(t)J_\pi(z) \textrm{ for all } t\in C.
\end{equation}
From \eqref{eq:Bessel} follows that
\begin{equation}\label{invariance}
T_{g,\pi,\theta_2}(u,v)=T_{g,\pi,\theta_1}(u,v)\overline{\pi}(x_{u,2}/x_{u,1})\overline{\pi}(x_{v,2}/x_{v,1}).
\end{equation}
when $g\notin B$ and \eqref{invariance} holds also in the case $g\in B$ which can be verified directly. In the following sections sets of biangular lines will be represented by row and column vectors of matrices $T_{g,\pi,\theta}$. Equation \eqref{invariance} implies that the choice of $\theta$ does not affect the structure of the resulting line system.
\end{remark}
\begin{remark}
If $q$ is even, by using perfectness of the field $\mathbb{L}$, it is possible to choose representatives of the circles $C_d$ to be $\theta=\{\sqrt{d}\mid d\in\mathbb{F}_q^*\}$. With this particular choice, 
\begin{displaymath}
T_{g,\pi}(u,v)=\begin{cases}
-\chi(\frac{au+dv}{c})J_\pi(\sqrt{vu}/c),&\text{if } g\notin B\\
\chi(abu)\delta(v-N(a)u),&\text{otherwise.}
\end{cases}
\end{displaymath}
In order to get rid of square roots in these expressions, the matrices $T_{g,\pi,}$ will be defined for $q$ even. Notice that $u^2=N(u)$ for all $u\in\mathbb{F}_q$. Let $T_{g,\pi}(v,u)=T_{g,\pi,\theta}(N(v),N(u))$ for all $v,u\in\mathbb{F}_q^*$. Explicitly by performing substitution we have
\begin{displaymath}
T_{g,\pi}(u,v)=\begin{cases}
-\chi(\frac{aN(u)+dN(v)}{c})J_\pi(vu/c),&\text{if } g\notin B\\
\chi(abN(u))\delta(N(v)-N(au)),&\text{otherwise}
\end{cases}
\end{displaymath}
for all $v,u\in\mathbb{F}_q^*$. The matrices $T_{g,\pi}$ define representations $(H(\mathbb{F}_q^*),\phi_{2,\pi})$ where homomorphisms $\phi_{2,\pi}$ are defined by $\phi_{2,\pi}(g)=T_{g,\pi}$ for all $g\in G$.
\end{remark}
\bigskip

\section{Sets of biangular lines related to representations of $\SL(2,\mathbb{F}_{2^{2k+1}})$}

This section is dedicated to real line systems $\mathcal{L}_\pi$ represented by row or column vectors of matrices $T_{g,\pi}$ related to representations $(H(\mathbb{F}_q^*),\phi_{2,\pi})$. In the first part of this section we will derive Propositions \ref{Norm}-\ref{BC} which express basic properties common to all such line systems. In particular, Proposition \ref{BC} gives a lower bound on coherence; line systems achieving this lower bound are always biangular.  We will show then, that a particular infinite family achieving the lower bound in Proposition \ref{BC} actually exists. This is done by estalishing  a connection through Hasse-Davenport relations between some of the functions $J_\pi$ and a certain additive character sums.

The usual $p$-norm of a vector $v\in H(\mathbb{F}_q^*)$ is denoted by $||v||_p$ where $1\leq p\leq\infty$. For $c\in\mathbb{F}_q$ define  $J_{\pi,c}(x)=J_\pi(cx)$ and $\chi_{c}(x)=\chi(cN(x))$ for all $x\in\mathbb{F}_q$. Some important basic properties of the functions $J_\pi$ which will be needed later are described in following proposition.
\begin{proposition}
\label{Norm}
For $q$ even and $\pi$ a nontrivial character of $C$, the following properties of the functions $J_\pi$ hold:
\begin{enumerate}
\item $J_{\pi}(x)\in\mathbb{R}$ and $J_\pi(x)=J_{\overline{\pi}}(x)$ for all $x\in\mathbb{F}_q^*$.
\item $\vert\vert J_\pi\vert\vert_2^2=1$.
\item $\vert\vert J_\pi^2\vert\vert_2^2=\frac{2}{q}$.
\item $\vert\vert J_\pi J_{\pi,c}\vert\vert_2^2=\frac{1}{q}$ when $c\neq0,1$.
\end{enumerate}
\end{proposition}
\begin{proof}
To prove the first identity observe that $S(x\sigma(t))=S(xt)$ and $\pi(t)=\overline{\pi(\sigma(t))}$ which gives
\begin{displaymath}
J_{\overline{\pi}}(x)=\sum_{t\in C}\chi(S(xt))\overline{\pi(t)}/q=\sum_{t\in C}\chi(S(x\sigma(t)))\overline{\pi(\sigma(t))}/q=\sum_{t\in C}\chi(S(xt))\pi(t)/q=J_\pi(x)\\
\end{displaymath}
and on the other hand
\begin{displaymath}
J_{\overline{\pi}}(x)=\sum_{t\in C}\chi(S(xt))\overline{\pi(t)}/q=\sum_{t\in C}\overline{\chi(S(xt))\pi(t)}/q=\overline{J_\pi(x)},
\end{displaymath}
for all $x\in\mathbb{F}_q^*$. The second identity is obvious since the matrices $T_{g,\pi}$ are orthogonal. For the third identity, fix  {
$
\arraycolsep=1pt\def\arraystretch{1}g_1=\left[
\scriptsize
\begin{array}{cc}
0&1\\
1&0
\end{array}
\right]
$
} and set $
\arraycolsep=1pt\def\arraystretch{1}g_2(d)=\left[
\scriptsize
\begin{array}{cc}
0&1\\
1&d
\end{array}
\right]
$ for $d\in\mathbb{F}_q$. Now $
\arraycolsep=1pt\def\arraystretch{1}g_2(d)g_1=g_3(d)=\left[
\scriptsize
\begin{array}{cc}
1&0\\
d&1
\end{array}
\right]
$, $T_{g_2(d),\pi}T_{g_1,\pi}=T_{g_3(d),\pi}$ and explicitly
\begin{displaymath}
Q(d):=T_{g_3(d),\pi}(1,1)=\chi(dN(0))J_\pi^2(0)+\sum_{x\in\mathbb{F}_q^*}J_\pi(x)^2\chi(dN(x))\textrm{ for all }d\in\mathbb{F}_q.
\end{displaymath}
Since $N(x)=x^2$ for all $x\in\mathbb{F}_q$, the functions $\chi(dN(t))/\sqrt{q}$ form an orthonormal basis of $H(\mathbb{F}_q)$ and therefore (by Parseval's theorem) $\vert\vert Q\vert\vert_{H(\mathbb{F}_q)}^2=q\vert\vert J_\pi^2\vert\vert_2^2$. On the other hand, the values of the function \textcolor{green}{$Q$} are known:
\begin{displaymath}
Q(d)=T_{g_3(d),\pi}(1,1)=\begin{cases}
-\chi(N(1)/d)^2 J_\pi(1/d)&\text{if }d\neq0\\
1&\text{if }d=0.
\end{cases}
\end{displaymath}
for all $d\in\mathbb{F}$. We conclude that $\vert\vert Q\vert\vert_{H(\mathbb{F}_q)}^2=1^2+\vert\vert J_\pi \vert\vert_2^2=2$ and thus $\vert\vert J_\pi^2\vert\vert_2^2=\frac{2}{q}$.

The fourth identity can be proved in the same way as the third identity by choosing  {
$
\arraycolsep=1pt\def\arraystretch{1}g_1(c)=\left[
\scriptsize
\begin{array}{cc}
0&c\\
c^{-1}&0
\end{array}
\right]
$
} with any fixed $c\neq 0,1$ to get for all $d\in\mathbb{F}_q$
\begin{displaymath}
Q_c(d):=T_{g_3(c,d),\pi}(1,1)=\chi(dN(0))J_\pi(0)J_\pi(c\cdot0)+\sum_{y\in\mathbb{F}_q^*}J_\pi(x)J_\pi(cx)\chi(dN(x))
\end{displaymath}
where  $g_3(c,d)=g_2(d)g_1(c)$. Explicitly
\begin{displaymath}
Q_c(d)=T_{g_3(c,d),\pi}(1,1)=\begin{cases}
-\chi((c^2+1)N(1)/d) J_\pi(c/d)&\text{if }d\neq0\\
0&\text{if }d=0.
\end{cases}
\end{displaymath}
Since $|\chi(y)|=1$ for all $y\in\mathbb{F}_q$, we have  $\vert\vert Q_c\vert\vert_2^2=0^2+\vert\vert J_\pi \vert\vert_2^2=1$ and thus $\vert\vert J_\pi J_{\pi,c}\vert\vert_2^2=\frac{1}{q}$.
\end{proof}

Now it is time to define line systems $\mathcal{L}_\pi$ in the case when $q$ is even by giving explicitly corresponding sets of representatives. 
\begin{definition}\label{LSyst}
The line system where each line is represented by a vector in the set
\begin{displaymath}
\Phi_\pi=\{\chi_aJ_{\pi,y}\mid a\in\mathbb{F}_q,y\in\mathbb{F}_q^*\}\cup\{\delta_y\mid y\in\mathbb{F}_q^*\}\subset H(\mathbb{F}_q^*).
\end{displaymath}
is denoted by $\mathcal{L}_\pi$.
\end{definition}
Each vector in $\Phi_\pi$ is a row vector in some matrix  $T_{g,\pi}$ where $g\in G$. On the other hand, any row vector of any matrix $T_{g,\pi}$ is parallel with some vector of the set $\Phi_\pi$.
\begin{proposition}
\label{CodeSize1}
If $q$ is even, $|\Phi_\pi|=q^2-1$ for all nontrivial characters $\pi$ and  $q>2$. All vectors in $\Phi_\pi$ are real and non-parallel.
\end{proposition}
\begin{proof}
The set $\Phi_\pi$ contains $q-1$ vectors of the form $\delta_y$ as well as at most $q(q-1)$ vectors of the form $\chi_aJ_{\pi,y}$. All these vectors are real according to part 1 of Proposition \ref{Norm}. The cardinality of $\Phi_\pi$ is thus at most $q^2-1$. It remains to show that cardinality of $\Phi_\pi$ is at least $q^2-1$ by showing that there are no parallel vectors in the description of the set $\Phi_\pi$. 

The multiplicative group $\mathbb{L}^*$ is a subgroup of $\GL(2,\mathbb{F}_q)$ where every element $z=b+c\rho\in\mathbb{L}^*$ is represented by the invertible matrix $g_z\in\GL(2,\mathbb{F}_q)$ where {
$
\arraycolsep=1pt\def\arraystretch{1}g_z=\left[
\scriptsize
\begin{array}{cc}
b&sc\\
c\textrm{ }&b+c
\end{array}
\right]
$
} and $s\in\mathbb{F}_q$ as in Section 2. Matrix $g_z$ is upper triangular exactly when $c=0$, in this case $g_z$ is a scalar matrix and hence $z\in\mathbb{F}_q^*\subset\mathbb{L}^*$. Matrices associated with elements $z\in C$ have $\Det(g_z)=N(z)=1$ and thus form a cyclic subgroup of $\SL(2,\mathbb{F}_q)$. Since $C\cap\mathbb{F}_q^*=\{1\}$, the only upper triangular matrix corresponding to elements of $C$ is the identity matrix. 

Let $z_0$ be a generator of $C$. From the group property of $C$ follows
\begin{displaymath}
\langle T_{g_{z_0},\pi}^r(u,:),T_{g_{z_0},\pi}^t(v,:)\rangle=T_{g_{z_0},\pi}^{r-t}(u,v) \textrm{ for all } 0\leq r,t<q+1.
\end{displaymath}
If $r=t$, then $T_{g_{z_0},\pi}^{r-t}$ is an identity matrix. In all other cases matrices $g_{z_0}^{r-t}$ are not upper triangular since $C\cap\mathbb{F}_q^*=\{1\}$. Thus, the absolute values of entries of matrices $T_{g_{z_0},\pi}^{r-t}$ are thus contained in the set $\{|J_\pi(x)|\mid x\in\mathbb{F}_q^*\}$. From part 3 of Proposition \ref{Norm} follows that $|J_\pi(x)|\leq \sqrt[4]{\frac{2}{q}}<1$ for all $x\in\mathbb{F}_q^*$. This can be used to deduce that all row vectors $T_{g_{z_0}^r,\pi}$, $0\leq r<q+1$ are on different lines. Thus, one of the vectors in the each pair $\{T_{g_{z_0^r},\pi}(v,:),-T_{g_{z_0^r},\pi}(v,:)\}$, $0\leq r<q+1, v\in\mathbb{F}_q^*$ is contained in the set $\Phi_\pi$,  proving that  $|\Phi_\pi|=q^2-1$.
\end{proof}

Our next goal is to determine a lower bound for the coherence of the line packings $\mathcal{L}_\pi$.
\begin{proposition}\label{BC}
Let $q$ be even. For all nontrivial characters $\pi$ of the circle $C$, $\mu_{\mathcal{L}_\pi}\geq\sqrt{2/q}$ holds. Equality can only be achieved if $\mathcal{L}_\pi$ is $\{0,\sqrt{2/q}\}$-angular.
\end{proposition}
\begin{proof}
From Proposition \ref{CodeSize1} follows that the angle set of $\mathcal{L}_\pi$ is $\mathcal{A}_\mathcal{L}=\{|J_\pi(x)|\mid x\in\mathbb{F}_q^*\}$. Absolute values of the functions $J_\pi$ can be estimated with the help of Proposition \ref{Norm}. Indeed, parts 2 and 3 of Proposition \ref{Norm} imply that the maximal absolute value must be at least $\sqrt{2/q}$ since from the H\"{o}lder inequality follows
\begin{displaymath}
2/q=\vert\vert J_\pi^4\vert\vert_1\leq\vert\vert J_\pi^2\vert\vert_1\vert\vert J_\pi^2\vert\vert_\infty=\vert\vert J_\pi^2\vert\vert_\infty.
\end{displaymath}
On the other hand, if the maximal absolute value of the function $J_\pi$ is exactly $\sqrt{2/q}$ then coordinatewise 
\begin{displaymath}
J_\pi(x)^4\leq2J_\pi(x)^2/q \textrm{ for all } x\in\mathbb{F}_q^*
\end{displaymath}
and equality occurs only if $J_\pi(x)\in\{0,\pm\sqrt{2/q}\}$ for all $x\in\mathbb{F}_q^*$. In this case $\mathcal{A}_\mathcal{L}=\{0,\sqrt{2/q}\}$ and the line system $\mathcal{L}_\pi$ is biangular.
\end{proof}
The main goal of this section is to show that there actually exists a nontrivial character $\pi$ of the group $C$ such that coherence of $\mathcal{L}_\pi$ achieves equality in Proposition \ref{BC} and in the Levenshteins second bound. If $q=2^{2k+1}$ then
\begin{equation}\label{res}
2^{2k+1}\equiv (-1)^{2k+1}=-1 \textrm{ mod } 3 
\end{equation}
and thus $q+1$ is always divisible by 3. Since the character group of $C$ is isomorphic to $C$, there exist exactly three characters $\pi$ such that $\pi^3=\pi_0$. One of these characters is obviously $\pi_0$ which defines the $q$-dimensional representation. The other two characters are of order 3. When one is $\pi$, other is always $\overline{\pi}$.  These two characters will define the same line system $\mathcal{L}_{\overline{\pi}}=\mathcal{L}_\pi$ and the same Bessel function $J_{\overline{\pi}}=J_\pi$ on $\mathbb{F}_q^*$.

We will show that characters of order 3 always produce line systems $\mathcal{L}_\pi$ with optimal properties. To prove optimality of these line systems, Gauss sums will be introduced in Definition \ref{Gauss} and some of their properties will be introduced in Propositions \ref{Gauss:triv} and \ref{HD}.
\begin{definition}\label{Gauss}
A Gauss sum is a sum of the form 
\begin{displaymath}
G(\chi,\pi)=\sum_{x\in\mathbb{F}_q^*}\chi(x)\pi(x)
\end{displaymath}
where $\chi$ is an additive character of the field $\mathbb{F}_q$ and $\pi$ is a multiplicative character of the field $\mathbb{F}_q$.  
\end{definition}
\begin{proposition}\label{Gauss:triv}%%%%%%%%%%%%%%%%%%%%
For a nontrivial additive character $\chi$, Gauss sums have the following properties
\begin{enumerate}
\item $|G(\chi,\pi)|=\sqrt{q}$ for $\pi$ nontrivial,
\item $G(\chi,\pi)=-1$ for $\pi$ trivial,
\item $G(\chi,\overline{\pi})=\pi(-1)\overline{G(\chi,\pi)}.$
\end{enumerate}
\end{proposition}
\noindent The properties of Gauss sums in Proposition \ref{Gauss:triv} are well known.
\begin{proposition}[Hasse--Davenport relations, \cite{H-D}]\label{HD}
Let $\mathbb{F}_q$ be a finite field with $q$ elements and let $\mathbb{L}/\mathbb{F}_q$ be its extension of degree $s$. Let $\chi,\rho$  be (nontrivial) additive and multiplicative characters of the field $\mathbb{F}_q$, respectively. Let $N,S$ be the norm and the trace from $\mathbb{L}$ to $\mathbb{F}_q$, respectively. Let $\chi'(z)=\chi(S(z))$ and $\rho'(z)=\rho(N(z))$ be additive and multiplicative characters of the field $\mathbb{L}$. Then 
\begin{displaymath}
(-1)^sG(\chi,\rho)^s=-G(\chi',\rho') \textrm{ (Hasse--Davenport lifting relation)}.
\end{displaymath}
Let $\pi$ be a multiplicative character of the field $\mathbb{F}_q$ of order $m$ dividing $q-1$. Then
\begin{displaymath}
\prod_{a=0}^{m-1} G(\chi,\rho\pi^a)=-\rho^{-m}(m)G(\chi,\rho^m)\prod_{a=0}^{m-1}G(\chi,\pi^a) \textrm{ (Hasse--Davenport product relation)}.
\end{displaymath}
\end{proposition}
\noindent Results of L. Carlitz concerning the evaluation of some additive character sums will prove to be useful. In fact, we will prove that when a character $\pi$ has degree 3, the values of the functions $J_\pi$ are directly related to  additive character sums presented in the next proposition.
\begin{proposition}[L. Carlitz, \cite{C}]
\label{Carlitz}
Let $\mathbb{F}_q$ be a field with $q=2^{2k+1}$ elements and $\chi$ its nontrivial additive character. Then
\begin{displaymath}
\sum_{x\in\mathbb{F}_q}\chi(ax^3+bx)\in\{0,\pm2^{k+1}\}
\end{displaymath}
while at least one of the parameters $a,b$ is nonzero.
\end{proposition}
\noindent Finally, two simple auxiliary lemmas will be needed.
\begin{lemma}
\label{lemma1}
Let $r\in\{1,2,3,\ldots\}$ and $n$ odd with $gcd(r,n)=1$. Let $f:\mathbb{Z}/n\mathbb{Z}\rightarrow\mathbb{C}$ with  be a function with the following two properties
\begin{enumerate}
\item $|f(k)|=r$ for all $k\in\mathbb{Z}/n\mathbb{Z}$.
\item $f(k)=\sum_{0\leq j\leq n-1}a_j e^{2\pi ijk/n}$ with integral coefficients $a_j$ for all $0\leq j\leq n-1$.
\end{enumerate}
Let $s:\mathbb{Z}/n\mathbb{Z}\rightarrow \{-1,1\}$ be any function. If the function $s(k)f(k):\mathbb{Z}/n\mathbb{Z}\rightarrow\mathbb{C}$ can be expressed in the form
\begin{equation}\label{eq:lemma}
s(k)f(k)=\sum_{0\leq j\leq n-1}b_j e^{2\pi ijk/n}
\end{equation}
with integral coefficients $b_j$ for all $0\leq j\leq n-1$, then the function $s(k)$ is necessarily a constant.
\end{lemma}
\begin{proof}
Assume that there exist functions $f(k)$ and nonconstant $s(k)$ satisfying the requirements of Lemma \ref{lemma1}. Assume further that equation \eqref{eq:lemma} holds with integral coefficients $b_j$ for all $0\leq j\leq n-1$. Let $\mathcal{S}=\{k\in\mathbb{Z}/n\mathbb{Z}\mid s(k)=-1\}$; since $s(k)$ is not constant, $|\mathcal{S}|<n$. Now we determine the $H(\mathbb{Z}/n\mathbb{Z})$-norm of the function $g(k)=(1-s(k))f(k)$ in two different ways. Direct computation gives
\begin{displaymath}
||g||^2_{H(\mathbb{Z}/n\mathbb{Z})}=\sum_{0\leq k<n}|(1-s(k))f(k)|^2=4|\mathcal{S}|r^2.
\end{displaymath}
Alternatively, the same norm can be computed on the Fourier side
\begin{displaymath}
||g||^2_{H(\mathbb{Z}/n\mathbb{Z})}=n\sum_{0\leq k<n}(a_j-b_j)^2=n\cdot s,\textrm{ where } s\in\mathbb{N}.
\end{displaymath}
This leads to a contradiction since the first value is not divisible by $n$.
\end{proof}
\begin{lemma}\label{intJ}
Let $q$ be even and $\pi$ be a nontrivial character of the unit circle $C\subset\mathbb{L}$ of order $3$. Then $J_\pi(x)$ is rational of the form $k/q$ for all $x\in\mathbb{F}_q$. 
\end{lemma}
\begin{proof}
Since the character $\pi$ is of order 3, $\pi(t)\in\{1,e^{2\pi i/3},e^{-2\pi i/3}\}$ for all $t\in C$. Additionally $\chi(S(xt))=\chi(S(x\sigma(t)))$ and $\pi(t)=\overline{\pi(\sigma(t))}$ for all $t\in C$, $x\in\mathbb{F}_q$. Also $\pi(1)=1$. For a generator $t_0$ of $C$ we have
\begin{displaymath}
J_\pi(x)=\sum_{t\in C}\chi(S(xt))\pi(t)/q=\chi(x)\pi(1)/q+\sum_{1\leq n\leq q/2}\chi(S(xt_0^n))(\pi(t_0^n)+\pi(t_0^{q+1-n}))/q.
\end{displaymath}
Every summand in the last sum is either $2/q$ or $-1/q$, which completes the proof.
\end{proof}
\noindent The following theorem is the main result of this paper.%%%%%%%%%%%%%%%%%%%%%%%%%%%%%%%%%%%%
\begin{theorem}\label{result}
Let $q=2^{2k+1}$ and $\pi$ be a nontrivial character of the unit circle $C\subset\mathbb{L}$ of order $3$. Then the line system $\mathcal{L}_\pi$ of Definition \ref{LSyst} is $(q-1)$-dimensional  and  $|\mathcal{L}_\pi|=q^2-1$. Additionally $\mathcal{L}_\pi$ is  biangular with $\mathcal{A}_{\mathcal{L}_\pi}=\{0,\sqrt{2/q}\}$.
\end{theorem}
\begin{proof}
The fact that $\mathcal{L}_\pi$ is $(q-1)$-dimensional follows directly from Definition \ref{LSyst} while cardinality of $\mathcal{L}_\pi$ were determined in Proposition \ref{CodeSize1}. It remains is to determine the angle set of line systems $\mathcal{L}_\pi$.

Let $\hat{\mathbb{L}}$ be the group of multiplicative characters of the field $\mathbb{L}$ and $\psi:\mathbb{L}^*\rightarrow\hat{\mathbb{L}}$ be some fixed isomorphism. The group $\mathbb{L}^{*}$ is a direct product of two subgroups $C$ and $\mathbb{F}_q^{*}$ which have orders $q+1$ and $q-1$ respectively. The characters in $\psi(C)\leq\hat{\mathbb{L}}$ are constant on the lines $l_t=t\mathbb{F}_q^{*}$ where $t\in C$ while the characters in $\psi(\mathbb{F}_q^{*})\leq\hat{\mathbb{L}}$ are constant on the circles $C_d$ where $d\in\mathbb{F}_q^{*}$. A character $\pi^*$ of the group $C$ is a restriction of the corresponding character $\pi_{*}\in\psi(C)$. This correspondence is a bijection. 

Let $\xi$ be a generator of the multiplicative group $\mathbb{F}_q^*$ then the character
\begin{equation}\label{form}
\rho(\xi^k)=e^{2\pi ki/(q-1)} \textrm{ for all } 0\leq k<q-1
\end{equation}
generates all multiplicative characters of $\mathbb{F}_q^*$. In this case character $\rho_{*}=\rho\circ N$ generates $\psi(\mathbb{F}_p^*)$. From now on it will be assumed that  the character $\rho$ is defined according to \eqref{form} with some fixed $\xi$.

Define the function $f:\mathbb{Z}/(q-1)\mathbb{Z}\rightarrow\mathbb{C}$ by $f(r):=G(\chi\circ S, \pi_{*}\rho_{*}^r)$. Since the characters $\rho_{*}^r$ are constant on the circles $C_d$, it is possible to rewrite the expression for $f(r)$ into the form
\begin{displaymath}
f(r)=\sum_{z\in\mathbb{L}^{*}}\chi(S(z))\pi_{*}(z)\rho_{*}^r(z)=\sum_{x\in\mathbb{F}_q^{*}}\rho_{*}^r(x)\sum_{t\in C}\chi(S(xt))\pi_{*}(t)=q\sum_{x\in\mathbb{F}_q^{*}}\rho^r(x)J_\pi(x)=q\sum_{1\leq k<q-1}e^{\frac{2\pi ikr}{q-1}}J_\pi(\xi^k),
\end{displaymath}
where $z=xt$ and $x=\xi^k\in\mathbb{F}_q^{*}$, $t\in C$. Thus, the function $f(r)$ can be expressed as a sum of different orthogonal components. The amplitude of each component  $e^{\frac{2\pi i\cdot k}{q-1}}$ is $qJ_\pi(\xi^k)$.

From now on assume that the character $\pi^*=\pi$ is of order $3$, which means that the corresponding character $\pi_{*}$ is also of order $3$. Because of \eqref{res}, this is possible only if $q=2^{2k+1}$.

It is possible to compute the exact value of $f(0)$ by using the fact that the character $\pi^*$ is constant on the lines $l_t$ and the fact that $\sum_{t\in C}\pi(t)=0$. Indeed,
\begin{displaymath}
f(0)=\sum_{t\in C}\pi(t)\sum_{x\in F_q^*}\chi(S(xt))=q-1-\sum_{t\in C\setminus\{1\}}\pi(t)=q.
\end{displaymath}
It is possible to use the Hasse--Davenport product relation on the characters $\pi_{*},\rho_{*}$ and $\chi\circ S$ in order to obtain
\begin{equation}\label{eq:4}
\prod_{a=0}^{2}G(\chi\circ S, \pi_{*}^{a}\rho_{*}^r)=-\rho^{-3r}(3)G(\chi\circ S,\rho_{*}^{3r})\prod_{a=0}^{2}G(\chi\circ S, \pi_{*}^{a}) \textrm{ when } r\neq0.
\end{equation}
Now we have to simplify \eqref{eq:4}.  On the left-hand side of \eqref{eq:4} notice that $f(r)=G(\chi\circ S, \pi_{*}^{1}\rho_{*}^r)=G(\chi\circ S, \pi_{*}^{2}\rho_{*}^r)$. The reason for that is because characters $\pi$ and $\pi^2=\overline{\pi}$ define the same Bessel function $J_\pi$ on $\mathbb{F}_q^*$ and hence the same function $f(r)$. Now, the left-hand side of \eqref{eq:4} becomes
\begin{displaymath}
\prod_{a=0}^{2}G(\chi\circ S, \pi_{*}^{a}\rho_{*}^r)=G(\chi\circ S,\rho_{*})f(r)^2  \textrm{ when } r\neq0.
\end{displaymath}
On the right-hand side notice that $\rho_{*}^{-3r}(3)=1$ since $1=3$ in fields of characteristic $2$. Additionally, $q=f(0)=G(\chi\circ S, \pi_{*})=G(\chi\circ S, \pi_{*}^{2})$. And finally, $G(\chi\circ S, \pi_{*}^{0})=-1$ by Prorosition \ref{Gauss:triv}.2 since $\pi_*^{0}$ is a trivial character. Now, the right-hand side of \eqref{eq:4} becomes
\begin{equation}
-\rho^{-3r}(3)G(\chi\circ S,\rho_{*}^{3r})\prod_{a=0}^{2}G(\chi\circ S, \pi_{*}^{a})=q^2G(\chi\circ S,\rho_{*}^{3r})  \textrm{ when } r\neq0.
\end{equation}
By using all described simplifications and by dividing both sides of \eqref{eq:4} by $G(\chi\circ S,\rho_{*})$ we obtain
\begin{equation}\label{eq:5}
f(r)^2=\frac{q^2G(\chi\circ S,\rho_{*}^{3r})}{G(\chi\circ S, \rho_{*}^r)} \textrm{ when } r\neq0.
\end{equation}
To further simplify \eqref{eq:5} we can use part 3 of Proposition \ref{Gauss:triv} and the fact that $\rho_{*}=\rho\circ N$ to get
\begin{displaymath}
\frac{q^2}{G(\chi\circ S, \rho_{*}^{r})}=G(\chi\circ S, \overline{\rho}_{*}^r)=G(\chi\circ S, \rho_{*}^{-r})  \textrm{ when } r\neq0.
\end{displaymath}
This transforms \eqref{eq:5} into
\begin{displaymath}\label{eq:HD}
f(r)^2=G(\chi\circ S,(\rho\circ N)^{3r})G(\chi\circ S, (\rho\circ N)^{-r})  \textrm{ when } r\neq0.
\end{displaymath}
By using the Hasse--Davenport lifting relation the last equation can be transformed into
\begin{displaymath}
f(r)^2=G(\chi,\rho^{3r})^2G(\chi, \rho^{-r})^2  \textrm{ when } r\neq0,
\end{displaymath}
and finally
\begin{equation}
f(r)=\pm G(\chi,\rho^{3r})G(\chi, \rho^{-r})  \textrm{ when } r\neq0.
\end{equation}
Now, we define a function $g:\mathbb{Z}/(q-1)\mathbb{Z}\rightarrow\mathbb{C}$ such that $g(r)=G(\chi,\rho^{3r})G(\chi, \rho^{-r})$ for all $r\in\mathbb{Z}/(q-1)\mathbb{Z}$, explicitly
\begin{displaymath}
g(r)=\sum_{x\in\mathbb{F}_q^*}\rho(x)^{3r}\chi(x)\sum_{y\in\mathbb{F}_q^*}\rho^{-r}(y)\chi(y).
\end{displaymath}
The variable change $y=tx^3$, $t\in\mathbb{F}_q^*$ in previous sum (the function $tx^3$ is a bijection as long as $q=2^{2k+1}$) gives
\begin{displaymath}
g(r)=\sum_{t\in\mathbb{F}_q^*}\sum_{x\in\mathbb{F}_q^*}\rho(x)^{3r}\rho(tx^3)^{-r}\chi(x+tx^3)=\sum_{t\in\mathbb{F}_q^*}\rho(t)^{-r}\sum_{x\in\mathbb{F}_q^*}\chi(x+tx^3)=\sum_{0\leq k <q-1}e^{-2\pi i kr}\sum_{x\in\mathbb{F}_q^*}\chi(x+\xi^kx^3)
\end{displaymath}
which is the Fourier transform of $g(r)$. Coefficients corresponding to frequencies $e^{-2\pi i kr}$ are given by the sums
\begin{displaymath}
a_{-k}=\sum_{x\in\mathbb{F}_q^*}\chi(x+\xi^kx^3)\in\{-1,2^{2+1}-1,-2^{k+1}-1\},
\end{displaymath}
where exact set of the values is determined according to Proposition \ref{Carlitz}. Notice that $1=g(0)\neq f(0)=q$. This can be adjusted by introducing the new  function $g_1(r)=g(r)+(q-1)\delta(x)$. We have
\begin{displaymath}
g_1(r)=\sum_{0\leq k <q-1}e^{-2\pi i kr}\sum_{x\in\mathbb{F}_q}\chi(x+\xi^kx^3) \textrm{ for all } 0\leq r<q-1.
\end{displaymath}
Finally, by comparing properties of functions $g_1(r)$ and $f(r)$ we observe
\begin{enumerate}
\item $g_{1}(0)=f(0)=q$ and $f(r)=\pm g_{1}(r)$ for all $r\in\mathbb{Z}/(q-1)\mathbb{Z}$.
\item $|g_{1}(r)|=q$ for all $r\in\mathbb{Z}/(q-1)\mathbb{Z}$ and $gcd(q,q-1)=1$.
\item $f(r)$ can be expressed in the form $f(r)=\sum_{0\leq j<q-1}a_je^{2\pi ijr/(q-1)}$ where $a_j=qJ_\pi(\xi^k)\in\mathbb{Z}$ for all $0\leq j<q-1$ where  by Lemma \ref{intJ} all values of $qJ_\pi$ are integral.
\item $g_1(r)$  can be expressed in the form $g_1(r)=\sum_{0\leq j<q-1}b_je^{2\pi ijr/(q-1)}$ where  $b_j=\sum_{x\in\mathbb{F}}\chi(x+\xi^{-k}x^3)\in\{0,\pm 2^{k+1}\}$ according to Proposition \ref{Carlitz}.
\end{enumerate}
From Properties 1-4 we conclude that the functions $g_1(r)$ and $f(r)$ satisfy the requirements of Lemma \ref{lemma1} and thus either $g_1(r)=f(r)$ or $g_1(r)=-f(r)$ for all $r$. Since $f(0)=g_1(0)=q$ we conclude that $g_1(r)=f(r)$ for all $r$ and
\begin{displaymath}
J_\pi(\xi^k)=\frac{1}{q}\sum_{x\in\mathbb{F}}\chi(x+\xi^{-k}x^3)\in\{0,\pm\sqrt{2/q}\} \textrm{ for all } 0\leq k<q-1
\end{displaymath}
determining the angle set of $\mathbb{L}_\pi$.
\end{proof}
\noindent Small-dimensional line systems coming from Theorem \ref{result} deserve special remarks.
\begin{remark}
If $q=8$, then Theorem \ref{result} gives a system of $63$ $\{0,1/2\}$-angular lines in $\mathbb{R}^7$ which is isomorphic to the line system generated by the roots of the lattice $E7$.
\end{remark}
\begin{remark}
If $q=32$ then a system of $1023$ $\{0,1/4\}$-angular lines in $\mathbb{R}^{31}$ given by Theorem \ref{result} is not unique. A nonisomorphic line system with the same parameters  was found by the author of this article via a (nonexhaustive) computer assisted search.
\end{remark}
\bigskip
\bigskip
\section{Sets of biangular lines related to representations of $\SL(2,\mathbb{F}_{3^{k}})$}
In this section we investigate complex line systems associated with $(q-1)/2$-dimensional representations of the groups $\SL(2,\mathbb{F}_q)$ with $q$ being odd. Values of Bessel functions associated with these representations are explicitly determined in \cite{U}. However, the corresponding line systems are not investigated anywhere and seem to remain unknown. 

When $q$ is odd, the circle $C$ contains an even number of elements and there exist a unique multiplicative character $\pi_2$ of $C$ of order 2. Let $\theta=\{y_d \mid y_d\in C_d \textrm{ for every } d\in\mathbb{F}_q^*\}$ be some system of representatives of circles $C_d$ and $f:\mathbb{F}_q^*\rightarrow\mathbb{L}^*$ a function defined by the equation $f(d)=x_d$. It turns out that the representation  $(H(\mathbb{F}_q^*),\phi_{\pi_2})$ can be further decomposed into invariant subspaces. 
\begin{proposition}[See \cite{repr}]
\label{Decomp}
Let $q$ be odd. If $u,v\in\mathbb{F}_q^*$ such that $uv\in NQR_q$, then $T_{g,\pi_2,\theta}(u,v)=0$ for all $g\in\SL(2,\mathbb{F}_q)$.
\end{proposition}
\begin{corollary}
The representation $(H(\mathbb{F}_q^*),\phi_{\pi_2})$ can be decomposed into two subrepresentations, namely $(H(QR_q),\phi_{\pi_2}^+)$ and $(H(NQR_q),\phi_{\pi_2}^-)$ where
\begin{displaymath}
\phi_{\pi_2}^+(g)=T_{g,\pi_2,\theta}^{+}\in\GL(H(QR_q)) \textrm{ such that } T_{g,\pi_2,\theta}^{+}(u,v)=T_{g,\pi_2,\theta}(u,v) \textrm{ for all } u,v\in QR_q
\end{displaymath}
and
\begin{displaymath}
\phi_{\pi_2}^-(g)=T_{g,\pi_2,\theta}^{-}\in\GL(H(NQR_q)) \textrm{ such that } T_{g,\pi_2,\theta}^{-}(u,v)=T_{g,\pi_2,\theta}(u,v) \textrm{ for all } u,v\in NQR_q.
\end{displaymath}
\end{corollary}
\begin{proposition}[Curtis, Shinoda \cite{U}]
\label{Unitary}
If $q$ is odd, then for all $z\in\mathbb{F}_q^*$
\begin{displaymath}
J_{\pi_2}(z)=-\pi_2'(z)G(\chi,\pi_2')(\chi(-2z)-\chi(2z))/q,
\end{displaymath}
where $\pi_2'$ is the unique multiplicative character of $\mathbb{F}_q^*$ of order 2.
\end{proposition}
\begin{remark}\label{angle}
Propositions \ref{Decomp} and \ref{Unitary} together with \eqref{eq:Bessel} makes it possible to explicitly compute absolute values of the function $J_{\pi_2}$. Indeed on the circles $C_d$ where $d\in QR_q$, if $x_d$ is any of two possible elements such that $x_d^2=d$, then
\begin{displaymath}
J_{\pi_2}(tx_d)=-\pi_2'(x_d)G(\chi,\pi_2)(\chi(-2x_d)-\chi(2x_d))\overline{\pi}_2(t)/q \textrm{ for all } t\in C.
\end{displaymath}
Since 
\begin{displaymath}
|\pi_2'(x_d)|=|\pi_2(t)|=1\textrm{, } |G(\chi,\pi_2)|=\sqrt{q} \textrm{ and }|\chi(-2x_d)-\chi(2x_d)|\in\{|2\textrm{sin}(2\pi k/p)|\mid k\in\mathbb{Z}/p\mathbb{Z}\}
\end{displaymath} 
for all $x_d\in\mathbb{F}_q^*$ and $t\in C$, the absolute values of the function $J_{\pi_2}$ on the circles $C_d$ where $d\in QR_q$ always belong to the set 
\begin{displaymath}
\mathcal{A}_{\pi_2}=\{|\frac{2\textrm{sin}(2\pi k/p)}{\sqrt{q}}| \mid k\in\mathbb{Z}/p\mathbb{Z}\}.
\end{displaymath}
On the other hand, Proposition \ref{Decomp}  implies that $J_{\pi_2}(z)=0$ if $z\in C_d$, where $d\in NQR_q$. By combining two previous cases we conclude that the absolute values of the function $J_{\pi_2}$ are always contained in the set $\mathcal{A}_\pi$.
\end{remark}

In Section 4 we defined line systems $\mathcal{L}_{\pi_2}$ in the case when $q$ is even through the sets of representatives $\Phi_\pi$. Line systems $\mathcal{L}_{\pi_2}$ in the case where $q$ is odd can be defined similarly by taking all row vectors of the matrices $T_{g,\pi_2,\theta}$ modulo phases.
\begin{definition}\label{LSyst2}
The line system where each line is represented by a vector in the set
\begin{displaymath}
\Phi_{\pi_2,\theta}=\{\chi_a (J_{\pi_2,y}\circ f)\mid a\in\mathbb{F}_q,y\in\mathbb{F}_q^*\}\cup\{\delta_y\mid y\in\mathbb{F}_q^*\}\subset H(\mathbb{F}_q^*).
\end{displaymath} 
is denoted by $\mathcal{L}_{\pi_2}$.
\end{definition}
According to Remark \ref{INV} the line system $\mathcal{L}_{\pi_2}$ doesn't depend on the choice of $\theta$. The following is an analogue to Proposition \ref{CodeSize1}.
\begin{proposition}\label{Expl}
If $q>3$ is odd, the set of representatives $\Phi_{\pi_2,\theta}$ defines a line system $\mathcal{L}_{\pi_2}$ consisting of $q^2-1$ complex lines in $H(\mathbb{F}_q^*)$. All vectors in $\Phi_{\pi_2,\theta}$ are non-parallel.
\end{proposition}
\begin{proof}
We show that the set $S\subset \SL(2,\mathbb{F}_q)$ where
\begin{displaymath}
S=\{\arraycolsep=1pt\def\arraystretch{1}\left[
\scriptsize
\begin{array}{cc}
1&0\\
0&1
\end{array}
\right]\}\cup \{\left[
\scriptsize
\begin{array}{cc}
0&-a^{-1}\\
a&d
\end{array}
\right]\mid a\in\mathbb{F}_q^*,d\in\mathbb{F}_q\}
\end{displaymath}
generates $q^2-1$ complex lines with representatives in $\Phi_{\pi_2,\theta}$. The set of matrices 
\begin{displaymath}
S_1=\{g_1g_2^{-1}|g_1\neq g_2, g_1,g_2\in S\}
\end{displaymath}
contains only matrices which are not in $B$. Now, we have that
\begin{equation}\label{eq:group}
 \langle T_{g_1,\pi_2,\theta}(u,:),T_{g_2,\pi_2,\theta}(v,:)\rangle=T_{g_1g_2^{-1},\pi_2,\theta}(u,v)
\end{equation}
for all $g_1,g_2\in G$ because of the group property of the matrices $T_{g,\pi_2,\theta}$ where $g\in G$. In particular, \eqref{eq:group} holds for all $g_1,g_2\in S$. When $g_1,g_2\in S$ and $g_1\neq g_2$, we have $g_1g_2^{-1}\in S_1$ and the absolute values of the angles between row vectors of matrices $T_{g_1,\pi_2,\theta}$ and $T_{g_2,\pi_2,\theta}$ are contained in the set $\mathcal{A}_{\pi_2}$ because of \eqref{Brepr} for all matrices $T_{g,\pi_2,\theta}$ where $g\notin B$. If $g_1=g_2$, we have that 
\begin{displaymath}
\langle T_{g_1,\pi_2,\theta}(u,:),T_{g_1,\pi_2,\theta}(v,:)\rangle=0
\end{displaymath} 
for all $u\neq v$ because of the unitarity of matrices $T_{g_1,\pi_2,\theta}$.

Since every element of the set $\mathcal{A}_{\pi_2}$ is strictly smaller than $1$ as long as $q>3$, we conclude that every row vector of every matrix corresponding to an element in the set $S$ defines a different line, together producing $q^2-1$ different lines in $H(\mathbb{F}_q^*)$. On the other hand, cardinality of $\Phi_{\pi_2,\theta}$ cannot exceed $q^2-1$ which means that cardinality of the set $\Phi_{\pi_2,\theta}$ is exactly $q^2-1$.\qedhere
\end{proof}
\noindent Now we are ready to combine Propositions \ref{Decomp} to \ref{Expl} into the main result of this section.
\begin{theorem}\label{result1}
Let $q=3^k$ with $k\geq2$. Then the line system $\mathcal{L}_{\pi_2}$ of Definition \ref{LSyst2} is $(q-1)$-dimensional and $|\mathcal{L}_{\pi_2}|=q^2-1$. Additionally, $\mathcal{A}_{\pi_2}=\{0,\sqrt{3/q}\}$, and half of the lines, denoted by $\mathcal{L}_{\pi_2}^+$, are entirely in the subspace $H(QR_q)$ while the other half, denoted by $\mathcal{L}_{\pi_2}^-$, are in the subspace $H(NQR_q)$. 
\end{theorem}
\begin{proof}
The angle set $\mathcal{A}_{\pi_2}$ is given in the Remark \ref{angle}. The cardinality of $\mathcal{L}_{\pi_2}$ can be determined from Proposition \ref{Expl}. Proposition \ref{Decomp}  can be used to conclude that half of the vectors of the each basis $T_{g,\pi_2,\theta}$ where $g\in S$ lies in $H(QR_q)$ while other half lies in $H(NQR_q)$.
\end{proof}
\begin{remark}
If  $q=9$ then Theorem \ref{result1} produces a system of 40 $\{0,1/\sqrt{3}\}$-angular lines in $\mathbb{C}^4$. An isomorphic line system is given (for example) in \cite{Be} and \cite{Hog}.
\end{remark}
\section{Acknowledgements}
I would like to thank Professor Patric \"Osterg\aa rd for valuable comments.

\end{document}